\documentclass{amsart}

\usepackage{amssymb}
\usepackage{graphicx}
\usepackage{hyperref}
\usepackage{color}
\usepackage{verbatim}

\newtheorem{theorem}{Theorem}
\newtheorem{lemma}[theorem]{Lemma}

\newtheorem{corollary}[theorem]{Corollary}
{\theoremstyle{definition}
\newtheorem{example}[theorem]{Example}
\newtheorem{remark}[theorem]{Remark}
}

\newcommand{\R}{\ensuremath{\mathbb{R}}}

\newcommand{\Z}{\ensuremath{\mathbb{Z}}}
\newcommand{\N}{\ensuremath{\mathbb{N}}}
\newcommand{\X}{\mathcal{X}}

\begin{document}

\title[Global centers and global injectivity]
{On the connection between global centers and global injectivity in the plane}

\author[F. Braun]{Francisco Braun}
\address{Departamento de Matem\'{a}tica, Universidade Federal de S\~{a}o Carlos, 
13565--905 S\~{a}o Carlos, S\~{a}o Paulo, Brazil}
\email{franciscobraun@dm.ufscar.br}

\author[J. Llibre]{Jaume Llibre}
\address{Departament de Matem\`{a}tiques, Universitat Aut\`{o}noma de Barcelona, 
08193 Be\-lla\-ter\-ra, Barcelona, Catalonia, Spain}
\email{jllibre@mat.uab.cat}



\date{\today}

\begin{abstract}
In this note we present a generalization of a result of Sabatini relating global injectivity and global centers. 
The shape of the image of the map is taking into account. 
Our proofs do not use Hadamard's theorem. 
\end{abstract}

\maketitle

\section{Introduction and statement of the main results}
Throughout our exposition $U\subset \R^2$ will be an open connected set.

Let $X,Y : U\to\R$ be $C^k$ functions for some $k\in \N$. 
We consider the vector field $\X = (X,Y)$, or equivalently the system of differential equations
\begin{equation}\label{generalsystem}
\dot{x} = X(x,y),\ \ \ \ \dot{y} = Y(x,y).
\end{equation}
Let $z_0$ be an isolated singular point of system \eqref{generalsystem}. 
We say that $z_0$ is a \emph{center} of \eqref{generalsystem} when there exists a neighborhood $V$ of $z_0$, $V\subset U$, such that each orbit of \eqref{generalsystem} in $V \backslash \{z_0\}$ is periodic. 
We define the \emph{period annulus} of center $z_0$, denoting it by $\mathcal{P}_{z_0}$, as the maximal open connected set $W \subset U$ such that $W\setminus \{z_0\}$ is filled with periodic orbits of $\mathcal{X}$. 
We say that the center is \emph{global} when $\mathcal{P}_{z_0} = U$. 
We say that the center is \emph{isochronous} when the orbits in $\mathcal{P}_{z_0}$ have the same period. 

When the singular point $z_0$ is non-degenerate, i.e. the determinant of the linear part of $\X$ in $z_0$ is different from zero, in order to have a center it is necessary that the eingenvalues of $D\X (z_0)$ are purely imaginary. 
In this case we will say that the center $z_0$ is \emph{non-degenerate}.

Let $H: U\to\R$ be a $C^{k+1}$ function. 
We say that $H$ is the \emph{Hamiltonian} of system \eqref{generalsystem} if 
\begin{equation*}\label{Hamiltonian}
X(x,y) = - H_y (x,y), \ \ \ \ Y(x,y) = H_x (x,y).
\end{equation*}
In this case we call system \eqref{generalsystem} the \emph{Hamiltonian system} associated to the Hamiltonian $H$. 
We also denote $\X = \nabla H^{\perp}$.

\smallskip

The following result provides a simple way to produce non-degenerate Hamiltonian centers. 
Let $f = (f_1, f_2): U \to\R^2$. 
We denote by $H_f: U\to \R$ the Hamiltonian defined by 
\begin{equation}\label{h3}
H_f(x,y) = \frac{f_1(x,y)^2 + f_2(x,y)^2}{2}, 
\end{equation}
for each $(x,y) \in U$.

\begin{lemma}\label{isolatedminimum}
Let $f = (f_1,f_2): U \to \R^2$ be a $C^2$ map. 
If $z_0 \in U$ is such that $\det D f(z_0) \neq 0$, then $z_0$ is a singular point of the Hamiltonian vector field $\nabla H_f^\perp$ if and only if $f(z_0) = (0,0)$. 
In this case, this singular point $z_0$ is a non-degenerate center of $\nabla H_f^{\perp}$ and also an isolated global minimum of $H_f$. 
In particular, if 
\begin{equation}\label{det2}
\det D f = {f_1}_x {f_2}_y - {f_2}_x {f_1}_y\neq 0
\end{equation}
in $U$, then the singular points of $\nabla H_f^{\perp}$ are non-degenerate centers and correspond to the zeros of $f$. 
\end{lemma}

In case the Jacobian determinant of $f$ in $U$ is a non-zero constant, it follows that the center $z_0$ is isochronous, see Theorem 2.1 of \cite{S1}. 
See also Theorem B of \cite{MV} for the characterization of the analytic Hamiltonian isochronous centers as being the ones such that locally the Hamiltonian has the form $H_f$, with $f$ having non-zero constant Jacobian determinant.

\smallskip 

When $f : \R^2 \to \R^2$ is a polynomial map satisfying \eqref{det2} and such that $f(0,0) = (0,0)$, Sabatini proved in \cite{S1} that $f$ is a global diffeomorphism if and only if the center $(0,0)$ of $\nabla H_f^{\perp}$ is global. 
See an application of this result to the real Jacobian conjecture in \cite{BGL}. 
The connection between injectivity of maps and centers also appears in \cite{S2}, where there are results relating the injectivity of $C^2$ maps $f : \R^2 \to \R^2$ having non-zero constant Jacobian determinant to the area of the period annulus of a center of $\nabla H_f^{\perp}$. 
In the same paper \cite{S2}, the injectivity of $f$ is also related to the property that some vector fields other than $\nabla H_f^{\perp}$ are complete, without assuming that the Jacobian determinant of $f$ is constant. 
In \cite{G} Gavrilov studied a connection between centers and injectivity in the complex context. 

\smallskip

The main aim of this note is the following extension of some of the above-mentioned results for $C^2$ maps defined in connected open sets of $\R^2$. 

\begin{theorem}\label{main}
Let $f: U \to \R^2$ be a $C^2$ map satisfying \eqref{det2} and $z_0 \in U$ such that $f(z_0) = (0, 0)$. 
The center $z_0$ of the Hamiltonian vector field $\nabla H_f^{\perp}$ is global if and only if (i) $f$ is injective and (ii) $f(U) = \R^2$ or $f(U)$ is an open disc centered at $(0,0)$. 
\end{theorem}

In case $f: \R^2 \to \R^2$ is a polynomial injective map, it follows that $f\left(\R^2\right) = \R^2$, see for instance \cite{BR}. 
Therefore our Theorem \ref{main} generalizes the above-mentioned result of \cite{S1}. 

\begin{corollary}\label{cmain}
Let $f: U \to \R^2$ be a $C^2$ map satisfying \eqref{det2} and $z_0 \in U$ such that $f(z_0) = (0,0)$. 
Then (i) $f$ is injective in $\overline{\mathcal{P}_{z_0}}$, where $\overline{\mathcal{P}_{z_0}}$ is the closure of $\mathcal{P}_{z_0}$ in $U$, and 
(ii) $f \left(\mathcal{P}_{z_0}\right) = \R^2$ or $f \left(\mathcal{P}_{z_0} \right)$ is an open disc centered at $(0,0)$. 
\end{corollary}

In case $U = \R^2$ and the Jacobian determinant of $f$ is $1$, the statement (i) of Corollary \ref{cmain} already appeared in \cite{S1} as Corollary 2.2. 

\smallskip

The following estimates the size of the period annulus $\mathcal{P}_{z_0}$.  

\begin{corollary}\label{cmain2}
Let $f: U \to \R^2$ be a $C^2$ map satisfying \eqref{det2} and $z_0 \in U$ such that $f(z_0) = (0, 0)$. 
Then $\mathcal{P}_{z_0}$ is the greatest open connected set containing $z_0$ such that (i) $f$ is injective in it and (ii) its image under $f$ is $\R^2$ or an open disc centered at $(0,0)$. 
\end{corollary} 

\smallskip

We observe that in our proofs it is not possible to use the classical Hadamard result of global invertibility of maps, that a local diffeomorphism $F : B \to B$, where $B$ is a Banach space, is a global one if and only if $F$ is proper. 
This is because our domain is just an open connected set, and our maps can be not surjective.

\smallskip 

We prove the results in section \ref{prsec} and present examples to them in section \ref{exsec}. 
We also study the special case where $H_f$ is polynomial in section \ref{final}.

\section{Proof of the results}\label{prsec} 
\begin{proof}[Proof of Lemma \ref{isolatedminimum}] 
Observe that $\nabla H_f^{\perp}(z_0) = (0,0)$ is equivalent to $Df(z_0) f(z_0) = (0,0)$. 
Since $D f(z_0)$ is invertible, it follows that $z_0$ is a singular point of $\nabla H_f^{\perp}$ if and only it is a zero of $f$. 

Assume so that $f(z_0) = (0,0)$. 
Since $f$ is locally injective, it follows that $f(z) \neq (0,0)$ for $z$ close enough to $z_0$, and so $z_0$ is an isolated global minimum of $H_f = \left(f_1^2 + f_2^2\right)/2$. 

The linear part of $\nabla H_f^\perp$ in $z_0$ is 
$$
D \nabla H_f^{\perp}(z_0) = \left(\begin{array}{cc}
- {f_1}_x {f_1}_y - {f_2}_x {f_2}_y  &  - {f_1}_y^2 - {f_2}_y^2 \\
{f_1}_x^2 + {f_2}_x^2  &  {f_1}_x {f_1}_y + {f_2}_x {f_2}_y.
\end{array}\right).
$$
Since $\det \left( D \nabla H_f^{\perp} \right) = \left(\det Df \right)^2 > 0$, we conclude that $z_0$ is a non-degenerate singularity and that the eigenvalues of $D \nabla H_f^{\perp}(z_0)$ are purely imaginary, because $\textnormal{tr} \left( D \nabla H_f^{\perp} \right) = 0$. 
Since the orbits of $\nabla H_f^{\perp}$ are contained in the level sets of $H_f$, and $z_0$ is an isolated minimum of $H_f$,  we conclude that $z_0$ is a center of this vector field.  \end{proof}

The proof of Theorem \ref{main} is a straightforward consequence of the following two lemmas.

\begin{lemma}
Let $f: U \to \R^2$ be an injective $C^2$ map satisfying \eqref{det2} and $z_0 \in U$ be such that $f(z_0) = (0,0)$. 
The center $z_0$ of the Hamiltonian vector field $\nabla H_f^{\perp}$ is global if and only if $f(U) = \R^2$ or $f(U)$ is an open disc centered at $(0,0)$.  
\end{lemma}

\begin{proof}
From hypothesis and from Lemma \ref{isolatedminimum}, the only singular point of $\nabla H_f^{\perp}$ is $z_0$. 
Thus from the definition of $H_f$ in \eqref{h3} we see that $z_0$ is the only point in the level set $H_f^{-1}\{0\}$, hence the non-singular orbits of $\nabla H_f^{\perp}$ are the connected components of the level sets of $H_f^{-1}\{h\}$ for $h > 0$, $h \in H_f (U)$. 
Clearly $h \in H_f (U)$ if and only if the circle 
$$
S_h = \left\{\sqrt{2 h} e^{i \theta}\ |\ \theta \in \R \right\} 
$$
intersects $f(U)$. 

Assume that $f(U)$ is $\R^2$ or a ball centered at $0$. 
Then $h \in H_f(U)$ if and only if $S_h \subset f(U)$. 
Therefore $H_f^{-1}\{h\}$ is the image of $S_h$ by $f^{-1}$. 
Thus $H_f^{-1}\{h\}$ is a topological circle. 
This proves that the non-singular orbits of $\nabla H_f^\perp$ are periodic. 
Hence the center $z_0$ is global. 

On the other hand, assume that the center $z_0$ is global. 
Let $y \in f(U)$, $y \neq (0,0)$, and set $h_y = H_f (f^{-1}(y))$. 
Since the orbits of $\nabla H_f^{\perp}$ are periodic, it follows that the connected components of $H_f^{-1}\{h_y\}$ are topological circles. 
Hence the image of each of them by $f$ is a topological circle contained in $S_{h_y}$. 
Therefore each image is the circle $S_{h_y}$ (and hence $H_f^{-1}\{h_y\}$ is connected). 
In particular, $S_{h_y} \subset f(U)$. 
Then we have just proved that for each $y \in f(U)$, the circle $S_{h_y}$ containing $y$ is contained in $f(U)$. 
As a consequence 
$$
f(U) = \{(0,0)\} \cup \bigcup_{h\in H_f(U)} S_h. 
$$
The set $H_f(U)$ is an interval of the form $[0, \ell)$, with $\ell = \infty$ or $\ell > 0$. 
Clearly $f(U) = \R^2$ if $\ell = \infty$, while if $\ell \in \R$, $f(U)$ is the open disc with radius $\ell$ centered at $(0,0)$. 
This finishes the proof of the lemma. 
\end{proof}

\begin{lemma}\label{fgtr}
Let $f: U \to \R^2$ be a $C^2$ map satisfying \eqref{det2} such that $\nabla H_f^{\perp}$ has a global center at the point $z_0 \in U$. 
Then $f$ is injective. 
\end{lemma}

\begin{proof}
Since $z_0$ is a global center, $z_0$ is the only singular point of $\nabla H_f^{\perp}$, corresponding, according to Lemma \ref{isolatedminimum}, to the level set $H_f^{-1}\{0\}$. 
Therefore for each $h \in H_f(U)$, $h \neq 0$, the level set $H_f^{-1}\{h\}$ is the union of periodic orbits of $\nabla H_f^{\perp}$. 

We \emph{claim that $H_f^{-1}\{h\}$ is connected}. 
Indeed, if $\gamma_1$ and $\gamma_2$ are two distinct periodic orbits of $\nabla H_f^{\perp}$ contained in $H_f^{-1}\{h\}$, they define an open topological annular region $\mathcal{A}$ whose boundary is $\gamma_1 \cup \gamma_2$. 
We take a $C^1$ injective curve $\lambda: [0,1] \to U$ such that $\lambda(0) \in \gamma_1$, $\lambda(1) \in \gamma_2$ and $\lambda((0,1))\subset \mathcal{A}$. 
Since $H_f(\lambda(0)) = H_f(\lambda(1)) = h$, it follows that the function $H_f \circ \lambda$ attains either its global maximum or minimum at a point $t_m \in (0,1)$. 
We consider the periodic orbit $\gamma_3$ of $\nabla H_f^{\perp}$ passing through $\lambda(t_m)$. 
This curve $\gamma_3$ separates $\mathcal{A}$ in two open connected regions $\mathcal{A}_1$ and $\mathcal{A}_2$. 
Clearly each $t \in (0,1)$ such that $\lambda(t) \in \gamma_3$ is an extreme of the function $H_f \circ \lambda$. 
Since the gradient of $H_f$ calculated at each point of $\gamma_3$ is different from zero, it follows that $\lambda((0,1))$ must be entirely contained in $\mathcal{A}_1$ or $\mathcal{A}_2$. 
But this is a contradiction, as the curve $\lambda$ connects $\gamma_1$ and $\gamma_2$. 
This contradiction proves the claim. 

We denote by $\gamma_h$ the orbit $H_f^{-1} \{h\}$. 
The claim proves in particular that $0 < h_1 < h_2$ if and only if the curve $\gamma_{h_1}$ is contained in the bounded region whose boundary is $\gamma_{h_2}$. 

\smallskip

To complete the proof it is enough to show that $f$ is injective in $\gamma_h$ for each $h \in H_f(U)$, $h \neq 0$. 
We consider the set 
$$
T = \{ h\in H_f(U),\ h\neq 0\ \ |\ f \textnormal{ is not injective in } \gamma_h \}. 
$$
It is enough to prove that $T$ is empty.

Suppose on the contrary that $T$ is not empty. 
Since $H_f(U) = [0, \ell)$, with $\ell = \infty$ or $\ell > 0$, the set $T$ is bounded from bellow. 
We let $h_{\alpha}$ be the infimum of $T$. 
Since $f$ is locally injective in $z_0$, it follows that $h_{\alpha} > 0$. 

We \emph{claim that $f$ is injective in $\gamma_{h_{\alpha}}$}. 
Indeed, if on the contrary there exist $a, b\in \gamma_{h_{\alpha}}$ with $a \neq b$ and  $f(a) = f(b)$, we consider neighborhoods $U_a$, $U_b$ and $V$ of $a$, $b$ and $f(a)$, respectively, with $U_a \cap U_b = \emptyset$, such that the maps $f|_{U_a} : U_a \to V$ and $f|_{U_b} : U_b \to V$ are diffeomorphisms. 
We let $C$ be the intersection of the segment connecting $(0,0)$ to $f(a)$ with the open set $V$, and we define the curves $C_a = f|_{U_a}^{-1}(C)$ and $C_b = f|_{U_b}^{-1}(C)$. 
The curves $C_a$ and $C_b$ are transversal sections to the flow of $\nabla H_f^{\perp}$, and both of them are contained in the compact region bounded by the curve $\gamma_{h_{\alpha}}$. 
In particular, for $h < h_{\alpha}$ near enough $h_{\alpha}$, the orbit $\gamma_h$ will cut $C_a$ and $C_b$. 
But then $f \left( C_a \cap \gamma_h \right) = f \left( C_b \cap \gamma_h \right)$, and hence $f$ is not injective in $\gamma_h$. 
This contradiction proves the claim.

Now from the definition of $h_{\alpha}$, there exists a sequence $\{h_n\}$, $h_n > h_{\alpha}$, that converges to $h_{\alpha}$ such that $f$ is not injective in $\gamma_{h_n}$. 
This means that for each $n$ there exist $a_n, b_n \in \gamma_{h_n}$ such that $a_n \neq b_n$ and $f(a_n) = f(b_n)$. 
Since $\{a_n\}$ and $\{b_n\}$ are contained in the compact set $\cup_n \gamma_{h_n}$, we can assume without loss of generality that there exist $a, b \in U$ such that $a_n \to a$ and $b_n \to b$ as $n \to \infty$. 
Since $h_n \to h_{\alpha}$, it follows that $a, b \in \gamma_{h_{\alpha}}$ and $f(a) = f(b)$. 
From the above claim, we have $a = b$. 
But as $f$ is locally injective in $a$, we obtain a contradiction with the assumptions that $a_n \neq b_n$, $f(a_n) = f(b_n)$, and $a_n \to a$ and $b_n \to b$. 
This contradiction proves that $T$ is empty and the lemma follows. 
\end{proof}

\begin{proof}[Proof of Corollary \ref{cmain}]
Let $g : \mathcal{P}_{z_0} \to \R^2$ be the map $f$ restricted to the open set $\mathcal{P}_{z_0}$. 
The center $z_0$ of the vector field $\nabla H_g^{\perp}$ defined in $\mathcal{P}_{z_0}$ is a global center. 
Thus from Theorem \ref{main} it follows that $g$ is injective and $g \left( \mathcal{P}_{z_0} \right) = \R^2$ or an open ball centered at the origin. 
This proves statement (ii) of the corollary and that $f$ is injective in $\mathcal{P}_{z_0}$. 

Let $F = \overline{\mathcal{P}_{z_0}}\backslash \mathcal{P}_{z_0}$ the boundary of $\mathcal{P}_{z_0}$ in $U$. 
Since for each $z \in F$ and for each $h \in H_f\left( \mathcal{P}_{z_0} \right)$ we have $H_f(z) > h$, it is enough to prove that $f$ is injective in $F$. 
This is quite similar to the last claim in the proof of Lemma \ref{fgtr}, therefore we give only the main idea of the proof. 
Suppose on the contrary the existence of $a, b \in F$, $a \neq b$, such that $f(a) = f(b)$. 
Let $U_a$, $U_b$ and $V$ neighborhoods of $a$, $b$ and $f(a)$, respectively, with $U_a \cap U_b = \emptyset$, such that the maps $f|_{U_a} : U_a \to V$ and $f|_{U_b} : U_b \to V$ are diffeomorphisms. 
Then acting as in the above proof, it is simple to get a contradiction with the injectivity of $f$ in $\mathcal{P}_{z_0}$. 
\end{proof}

\begin{proof}[Proof of Corollary \ref{cmain2}]
From Corollary \ref{cmain}, $\mathcal{P}_{z_0}$ satisfies (i) and (ii). 

Given an open connected set $V \subset U$ satisfying (i) and (ii), we apply Theorem \ref{main} to $f|_V : V \to \R^2$ obtaining that the orbits of $\nabla H_f^{\perp}$ intersecting $V$ are periodic and are contained in $V$. 
Thus $V \subset \mathcal{P}_{z_0}$. 
This finishes the proof of the corollary. 
\end{proof}

\section{Examples}\label{exsec}

\begin{example}
Let $f = (f_1, f_2) : \R^2 \to \R^2$ be defined by $f_1(x,y) = e^x - 1$, $f_2(x,y) = y$. 
We have $\det Df(x,y) = e^x$, hence $f$ satisfies \eqref{det2}. 
Moreover, $f$ is clearly injective, the image of $f$ is the set $(-1, \infty) \times \R$ and $f(0,0) = (0,0)$. 

From Theorem \ref{main}, the center $(0,0)$ is not global. 
From Corollary \ref{cmain2}, the image of its period annulus $\mathcal{P}_{(0,0)}$ under $f$ is the open ball centered at $(0,0)$ with radius $1$, that we denote by $B_1$. 
Thus $\mathcal{P}_{(0,0)} = f^{-1} \left(B_1\right) = \{ (x,y) \in \R^2\ |\ y^2 < e^x (2 - e^x) \}$. 
\end{example}

In the next example we present a global injective non-polynomial map $f$ in $\R^2$ with $f(0,0) = (0,0)$ which produces a polynomial Hamiltonian $H_f$. 
The center $(0,0)$ is a non-global isochronous center although $f$ is globally injective. 
\begin{example}\label{Nopolynomial}
Let $f = (f_1, f_2)$ be defined by
$$
f_1(x,y) = \frac{x}{\sqrt{1 + x^2}}, \ \ \ \ f_2(x,y) = \frac{x^2 + \left(1 + x^2\right)^2 y}{\sqrt{1 + x^2}}. 
$$
It is easy to see that the Jacobian determinant of $f$ is constant and equal to $1$ and that $(0,0)$ is the only zero of $f$. 
Thus $(0,0)$ is an isochronous center of $\nabla H_f^\perp$. 
Moreover, observe that 
$$
H_f(x,y) = \frac{\left(1 + x^2\right)^3}{2} y^2 + x^2 \left(1 + x^2\right) y + \frac{x^2}{2} 
$$ 
is a polynomial such that $H_f^{-1}\{1/2\}$ is an unbounded disconnected set. 
Hence $(0,0)$ is not a global center. 
This example has already appeared in \cite{CMV}. 

In Figure \ref{P4} we use the program \emph{P4}, see \cite{DLA}, to draw the separatrix skeleton of the Poincar\'e compactification of the vector field $\nabla H_f^\perp$ in the Poincar\'e disc. 
Observe that the infinite singular points in the $y$ direction are formed by two degenerate hyperbolic sectors. 
And the infinite singular points in the $x$ direction are formed by two non-degenerate hyperbolic sectors and two parabolic sectors. 
See section \ref{final}. 
\begin{figure}[!h]
\begin{center}
\includegraphics[scale=.4]{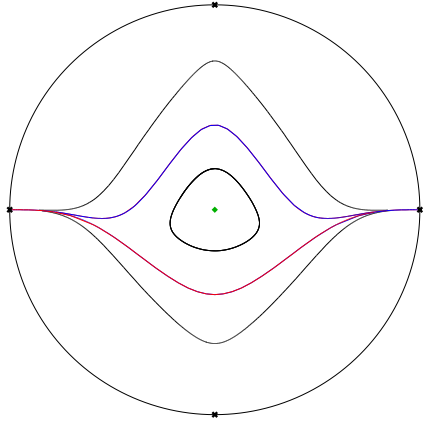}
\end{center}
\caption{Phase portrait of $\nabla H_f^\perp$ in the Poincar\'e disc.}
\label{P4}
\end{figure}
\end{example}

\begin{example}
Let $f = (f_1, f_2) : \R^2 \to \R^2$ be defined by $f_1(x,y) = e^x \cos y - 1$, $f_2(x,y) = e^x \sin y$. 
We have $\det D f(x,y) = e^{2 x}$. 
Moreover, the points $z_k = (0, 2 k \pi)$, $k \in \Z$, are the points that annihilate $f$. 
Therefore, the centers of $\nabla H_f^{\perp}$ are the points $z_k$, $k \in \Z$. 

We will estimate the period annulus $\mathcal{P}_{z_k}$ of each center $z_k$. 

Observe that $f(\R^2) = \R^2 \backslash \{(-1, 0)\}$, thus the biggest ball centered at $(0,0)$ contained in $f(\R^2)$ is $B_1$. 
In order that a point $(x,y)$ be such that $f(x,y) \in B_1$, it is necessary that $\cos y > 0$, which happens in the intervals $\left( (4 k - 1)\pi/2, (4 k + 1) \pi/2 \right)$, $k \in \Z$. 

It is easy to see that $f$ is injective in each of the sets $\R \times ((4 k - 1) \pi/2, (4 k + 1)\pi/2 )$, $k \in \Z$. 

Thus the exact set $\mathcal{P}_{z_k}$ is from Corollary \ref{cmain2} the set satisfying $f_1(x,y)^2 + f_2(x,y)^2 < 1$, with $y\in \left( (4 k - 1)\pi/2, (4 k + 1) \pi/2 \right)$. 
Straightforward calculations show that this is the set 
$$
\mathcal{P}_{z_k} = \left\{ (x, y) \in \R^2\ |\ e^x < 2 \cos y,\ (4 k - 1)\pi < 2 y < (4 k + 1) \pi \right\}. 
$$ 

Since $2 H_f = f_1^2 + f_2^2$, it follows that the connected components of the level sets $H_f^{-1}\{h\}$ with $h < 1/2$ give the periodic orbits of each center, and the connected components of the level set $H_f^{-1}\{1/2\}$ give the boundary of the period annulus of each center. 
Finally, it is simple to see that the level sets $H_f^{-1}\{h\}$ with $h > 1/2$ are connected. 
An overview of the level sets of $H_f$ in the plane can be seen in Figure \ref{fig1}. 
\begin{figure}[!h]
\begin{center}
\includegraphics[scale=.9]{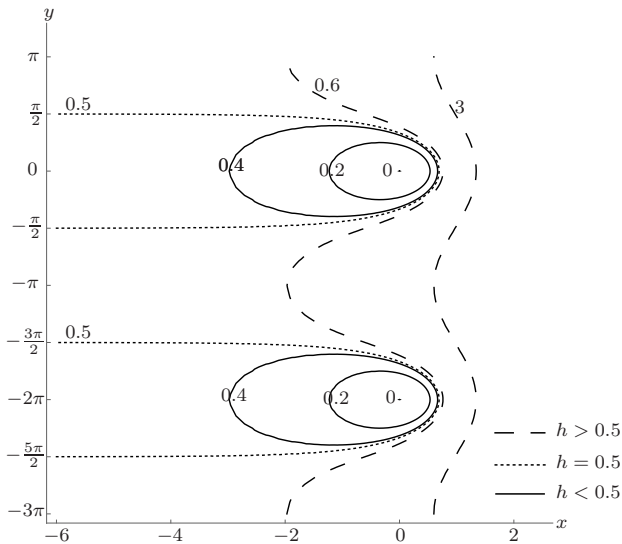}
\end{center}
\caption{The level sets of $H_f$.}
\label{fig1}
\end{figure}
\end{example}

Next example presents a non-injective polynomial map in $\R^2$ producing two centers. 

\begin{example} 
Let $g = (g_1, g_2): \R^2\to\R^2$ be the Pinchuk map as defined in \cite{C}. 
The image $g(\R^2)$ does not contain the points $(0,0)$ and $(-1, -163/4)$. 
Moreover, all the points of the curve $\left(P(s) , Q(s) \right)$ defined by 
$$
P(s) = s^2 - 1,\ \ \ \ Q(s) =  - 75 s^5 + \frac{345}{4} s^4 - 29 s^3 + \frac{117}{2} s^2 - \frac{163}{4}, 
$$
$s\in\R$, with the exception of $(0,0)$ and $(-1, -163/4)$, have exactly one inverse image under $g$. 
All the other points of $\R^2$ have two inverse images. 
The curve $\left( P(s), Q(s) \right)$ crosses the $y$-axis in $y = 0$ and in $y = 208$. 
For details on these results, see \cite{C}. 

We consider $f : \R^2 \to \R^2$ defined by translating the Pinchuk map as follows 
$$
f(x,y) = \big(g_1(x,y), g_2(x,y) - 200 \big). 
$$
Let $z_0^1$ and $z_0^2$ be the two elements of the set $f^{-1}\{(0,0)\} = g^{-1} \{(0,200)\}$. 
From Lemma \ref{isolatedminimum} the points $z_0^1$ and $z_0^2$ are centers of $\nabla H_f^{\perp}$. 

Since $(0, - 200)$ and $(-1, -163/4 - 200 )$ are the only points not contained in $f(\R^2)$, the greatest open ball centered at $(0,0)$ contained in $f(\R^2)$ is $B_{200}$. 
Moreover, from the properties of the Pinchuk map mentioned above, there exists an entire curve with just one inverse image under $f$ in this ball. 
All the other points have two pre-images. 
We consider $B_r$ the greatest ball centered at $(0,0)$ such that all its points have two inverse images under $f$. 
The inverse image of $B_r$ gives two open sets. 
One of them, say the one containing $z_0^1$, is the entire period annulus of the center $z_0^1$. 
The other open set is properly contained in the period annulus of the center $z_0^2$. 
This period annulus is mapped bijectively onto the open ball $B_{200}$. 
\end{example}

\section{The polynomial case}\label{final}

In this section given a polynomial vector field $\mathcal{X}$, we denote by $p(\mathcal{X})$ the \emph{Poincar\'{e} compactification of $\mathcal{X}$}. 
For details we refer the reader to chapter 5 of \cite{DLA}. 
As usual we call the singular points of $p(\mathcal{X})$ located in the equator of the Poincar\'e sphere $\mathbb{S}^2$ the \emph{infinite singular points} of $\mathcal{X}$. 
The other singular points we call \emph{finite singular points}. 

For a center $z_0$ of a polynomial vector field $\mathcal{X}$ we use the following classification of Conti, see \cite{Co}. 
We  say that the center $z_0$ is of \emph{type A} if $\partial \mathcal{P}_{z_0} = \emptyset$, i.e. the center is global, of \emph{type B} if $\partial \mathcal{P}_{z_0} \neq \emptyset$ and $\partial \mathcal{P}_{z_0}$ is unbounded and does not contain finite singular points, of \emph{type C} if $\partial \mathcal{P}_{z_0}$ contains finite singular points and is unbounded, and of \emph{type D} if $\partial \mathcal{P}_{z_0}$ contains finite singular points and is bounded. 
We remark that $\partial \mathcal{P}_{z_0}$ can never be a periodic orbit $\gamma$ of $\mathcal{X}$, otherwise let $\pi$ be the return Poincar\'{e} map defined in a transversal section $S$ through $\gamma$. 
Since $\pi$ is analytic and it is the identity map in the portion of $S$ contained in $\mathcal{P}_{z_0}$, it follows that it must be the identity in $S$, a contradiction with the fact that $\gamma$ is the boundary of $\mathcal{P}_{z_0}$. 

Let $q$ be an infinite singular point of the polynomial vector field $\mathcal{X}$ and $h$ be a hyperbolic sector of $q$ in the Poincar\'e sphere. 
We say that $h$ is \emph{degenerate} if its two separatrices are contained in the equator of $\mathbb{S}^2$. 
Otherwise we say that $h$ is \emph{non-degenerate}. 

\smallskip

In the following we give more equivalences to the injectivity of $f$ in case the Hamiltonian $H_f$ is polynomial. 

\begin{theorem}\label{polynomial}
Let $f : \R^2 \to \R^2$ be a $C^2$ map satisfying \eqref{det2} and $z_0 \in \R^2$ such that $f(z_0) = (0,0)$. 
If $H_f$ is polynomial the following statements are equivalent: 
\begin{itemize}
\item[(a)] $f$ is injective and $f\left(\R^2\right) = \R^2$ or $f\left(\R^2\right)$ is an open ball centered at $(0,0)$. 

\item[(b)] The center $z_0$ of $\nabla H_f^{\perp}$ is of type A. 

\item[(c)] The center $z_0$ of $\nabla H_f^{\perp}$ is not of type B. 

\item[(d)] The Hamiltonian vector field $\nabla H_f^\perp$ has no infinite singular points or each of them is formed by two degenerate hyperbolic sectors. 
\end{itemize}
\end{theorem}

\begin{proof}
Statements (a) and (b) are equivalent from Theorem \ref{main}. 
Moreover, since from Lemma \ref{isolatedminimum} the finite singular points of $\nabla H_f^\perp$ are centers, it follows that $\partial \mathcal{P}_{z_0}$ does not contain finite singular points. 
Therefore $\nabla H_f^\perp$ can not have centers of type C or D. 
Hence (b) is also equivalent to (c). 
It is also clear that (b) implies (d). 

Finally if the center $z_0$ is of type B, it follows that $\nabla H_f^\perp$ has at least one unbounded orbit, and thus there exist an infinite singular point without a degenerate hyperbolic sector. 
Hence (d) implies (c). 
This finishes the proof. 
\end{proof}

\begin{remark}
We remark that the assumption on the shape of $f\left(\R^2\right)$ in statement (a) of Theorem \ref{polynomial} is essential in general. 
Recall the above Example \ref{Nopolynomial}. 

If $f$ is assumed to be polynomial, then we can drop this hypothesis, as polynomial injective maps are onto, from \cite{BR}. 
\end{remark}

\section*{Acknowledgements}
The first author is partially supported by a BPE-FAPESP grant number 2014/ 26149-3. 
The second author is partially supported by a MINECO grant number MTM2013-40998-P, an AGAUR grant number 2014SGR 568 and two FP7-PEOPLE-2012-IRSES grants numbers 316338 and 318999. 
Both authors are also partially supported by a CAPES CSF--PVE grant 88881. 030454/ 2013-01 from the program CSF-PVE.

\end{document}